\documentclass[10pt]{amsart}
\theoremstyle{definition}
\usepackage{indentfirst, amssymb,amsmath,amsthm}
\usepackage{csquotes}
\usepackage{hyperref}
\usepackage{mathrsfs}
\newtheorem*{theoA}{Theorem A}
\newtheorem*{theoB}{Theorem B}
\newtheorem*{theoC}{Theorem C}

\newtheorem{theorem}{Theorem}[section]
\newtheorem{lem}{Lemma}[section]

\newtheorem{defi}{Definition}[section]

\newcommand{\be}{\begin{equation}}
\newcommand{\ee}{\end{equation}}
\newcommand{\beas}{\begin{eqnarray*}}
\newcommand{\eeas}{\end{eqnarray*}}
\newcommand{\bea}{\begin{eqnarray}}
\newcommand{\eea}{\end{eqnarray}}

\numberwithin{equation}{section}
\begin{document}
\title[Uniqueness of Meromorphic Functions]{Further Investigations on Weighted Value Sharing and Uniqueness of Meromorphic Functions}
\date{}
\author[S. Saha, A. K. Pal and S. Roy]{Sudip Saha$^{1, 2}$, Amit Kumar Pal$^{3, 4}$ and Soumon Roy$^{5}$}
\date{}
\address{$^{1}$Department of Mathematics, Ramakrishna Mission Vivekananda Centenary College, Rahara,
West Bengal 700 118, India.}
\email{sudipsaha814@gmail.com}
\address{$^{2}$Assistant Professor, Department of Mathematics, Brainware University, 398, Ramkrishnapur Road, Jagadighata Market, Barasat, Kolkata 700125, W.B., India}
\email{ss.math@brainwareuniversity.ac.in}
\address{$^{3}$Department of Mathematics, University of Kalyani, Kalyani, West Bengal 741 235, India.}
\address{$^{4}$Assistant Professor, Department of Mathematics, Rajiv Gandhi National Institute of Youth Development, Bangalore to Chennai National Highway, Beemanthangal,  Sri Ram Nagar, Sriperumbudur, Nemili, Tamil Nadu 602105.}
\email{mail4amitpal@gmail.com}
\address{$^{5}$Department of Mathematics, Ramakrishna Mission Vivekananda Centenary College, Rahara,
West Bengal 700 118, India.}
\email{rsoumon@gmail.com}
\maketitle
\let\thefootnote\relax
\maketitle
\begin{abstract}
In this short manuscript, we will put some light on the different outcomes when two non-constant meromorphic functions share a value with prescribed weight two.
\end{abstract}
\vspace{1cm}
\section{Introduction, Definition and Results}
In this paper, by meromorphic function we will always mean meromorphic functions in the complex plane. It will be convenient to let $E$ to denote any set of positive real numbers of finite linear measure, not necessarily same at each occurrence. For any non-constant meromorphic function $h(z)$ we denote by $S(r,h)$ any quantity satisfying 
$$S(r,h)=o(T(r,h)),$$
as $r \to \infty, r \not \in E.$\\
We shall adopt the standard notation of the Nevanlinna theory of meromorphic function as described in (\cite{wkh}). However we shall discuss some definitions and notations which will be needed in the sequel.\\
Let, $f$ and $g$ be two non-constant meromorphic functions and let $a$ be a finite complex number. We say that $f$ and $g$ share the value $a$ CM (counting multiplicities), provided that $f-a$ and $g-a$ have the same zeros with same multiplicities. Similarly we say that $f$ and $g$ share the value $a$ IM (ignoring multiplicities), provided that $f-a$ and $g-a$ have the same zeros, but the multiplicities are not taken into account. In addition to this we say that $f$ and $g$ share $\infty$ CM(IM), if $\frac{1}{f}$ and $\frac{1}{g}$ share $0$ CM(IM). 
\begin{defi}
We denote by $N(r,a;f|=1)$ the counting function of simple $a$-points of $f$.
\end{defi}
\begin{defi}
If $k$ be a positive integer, we denote by $\overline{N}(r,a;f| \geq k)$ the counting function of those $a$-points of $f$, whose multiplicity is greater than or equal to $k$.
\end{defi}
\begin{defi}
We denote by $N_{2}(r,a;f)$ the sum $\overline{N}(r,a;f)+ \overline{N}(r,a;f| \geq 2)$.
\end{defi}
\begin{defi}
Let $f$ and $g$ be two non-constant meromorphic functions such that $f$ and $g$ share $(a,0)(a \in \mathbb{C}\cup \{\infty\})$. Let $z_0$ be an $a$-point of $f$ with multiplicity $p$, an $a$-point of $g$ with multiplicity $q$. We denote by $\overline{N}_{L}(r,a;f)$ the reduced counting function of those $a$-points of $f$ and $g$, where $p>q$. \par
we denote by $N_{E}^{1)}(r,a;f)$ the counting function of those $a$-points of $f, g$ where $p=q=1$ and by $\overline{N}_{E}^{(2}(r,a;f)$ the reduced counting function of those $a$-points of $f$ and $g$ where $p=q\geq 2.$ In the same way we can define $\overline{N}_{L}(r,a;g), N_{E}^{1)}(r,a;g), \overline{N}_{E}^{(2}(r,a;g).$
\end{defi}
\begin{defi}
Let, $f, g$ share a value $a$ IM, we denote by $\overline{N}_{*}(r,a;f,g)$ the counting function of those $a$-points of $f$ whose multiplicities are not equal to the multiplicities of the corresponding $a$-points of $g$, where each $a$-point is counted only once.\\
Clearly $\overline{N}_{*}(r,a;f,g)=\overline{N}_{*}(r,a;g,f)$ and $\overline{N}_{*}(r,a;f,g)=\overline{N}_{L}(r,a;f)+\overline{N}_{L}(r,a;g)$.
\end{defi}
\begin{defi}
Let $k$ be a non-negative integer or infinity. For $a \in \mathbb{C} \cup \{\infty \}$ we denote by $E_{k}(a;f)$ the set of all $a$-points of $f$ where an  $a$-point of multiplicity $m$ is counting $m$ times if $m \leq k$ and $k+1$ times if $m > k$.\par
If $E_{k}(a;f)=E_{k}(a;g)$, then we say that $f$ and $g$ share the value $a$ with weight $k$. We sometimes also write $f, g$ share $(a,k)$ to mean that $f, g$ share the value $a$ with weight $k$. 
\end{defi}
The following is a very well known and important result in the uniqueness theory of meromorphic functions and had been proved by a number of authors in (\cite{mf, xhh, mr, yahu, yi}). This has a wide range of application in the uniqueness theory of meromorphic functions.
\begin{theoA}
If $f$ and $g$ share $1$ CM, one of the following three cases hold:
\begin{itemize}
\item[(i)] $T(r,f) \leq N_{2}(r,0;f)+ N_{2}(r,0;g)+ N_{2}(r,\infty;f)+ N_{2}(r,\infty;g)+S(r,f)+S(r,g),$ same inequality hold for $T(r,g)$.
\item[(ii)] $f \equiv g$,
\item[(iii)] $f \cdot g \equiv 1$.
\end{itemize}
\end{theoA}
Being motivated with this result Prof. I. Lahiri in (\cite{il}), tried to relax the nature of sharing to some lower weight and obtained the following two results:
\begin{theoB}\label{rth1.1}
Let $f, g$ share $(1,2)$. Then one of the following cases hold:
\begin{itemize}
\item[(i)]  $T(r,f) \leq N_{2}(r,0;f)+ N_{2}(r,0;g)+ N_{2}(r,\infty;f)+ N_{2}(r,\infty;g)+S(r,f)+S(r,g),$ same inequality hold for $T(r,g)$.
\item[(ii)] $f \equiv g$.
\item[(iii)] $f \cdot g = 1$.
\end{itemize}
\end{theoB}

\begin{theoC}\label{rth2}
Let $f,g$ share $(1,2)$ and $a(\neq 0, 1, \infty)$ be a complex number. Then one of the following cases hold:
\begin{itemize}
\item[(i)] $2T(r)  \leq  N_{2}(r, \infty; f)+ N_{2}(r, \infty; g)+ N_{2}(r, 0;f)+ N_{2}(r, 0;g)+N_{2}(r, a)+ S(r,f) +S(r,g)$, 
where $N_2(r,a)=\max \{N_2(r, a;f), N_2(r, a;g)\}$ and $T(r) = \max \{T(r,f), T(r,g)\}$.
\item[(ii)] $f \equiv g$.
\item[(iii)] $f \cdot g = 1$.
\item[(iv)] $\frac{1}{f}+\frac{1}{g}=2$
\item[(v)] $f+g= 2$
\item[(vi)] $ g(a-f)=a^2$
\item[(vii)] $ f(a-g)=a^2.$
\end{itemize}
\end{theoC}
Inspired by the Theorem B and Theorem C, we further tried to investigate the uniqueness of two meromorphic functions with weight 2. In the following we are stating the main result of our manuscript. 
\section{Main Results}
\begin{theorem}\label{th1}
Let, $f, g$ be two non-constant meromorphic functions which share $(1,2)$ and $a(\neq 0, 1, \infty), b(\neq a, 0, 1, \infty)$ be two complex numbers. Then one of the following cases holds:
\begin{itemize}
\item[(i)] $3T(r)  \leq  N_{2}(r, \infty; f)+ N_{2}(r, \infty; g)+ N_{2}(r, 0;f)+ N_{2}(r, 0;g)+N_{2}(r, a)+N_{2}(r, b)+ S(r,f) +S(r,g)$, 
where $N_2(r,a)=\max \{N_2(r, a;f), N_2(r, a;g)\}$ and similarly $N_2(r,b)$ is defined. Also, $T(r) = \max \{T(r,f), T(r,g)\}$.
\item[(ii)] $g(a-f)=a^{2}$. 
\item[(iii)] $g(a-f)=ab .$
\item[(iv)] $\frac{1}{f}+\frac{1}{g}=2.$
\item[(v)] $ \frac{a}{f}+\frac{b}{g}=1.$
\item[(vi)] $ (f-a)(g-a)=a(a-b).$
\item[(vii)] $ (f-a)(g-b)=b(a-b).$
\item[(viii)] $g(b-f)=ab .$
\item[(ix)] $g(b-f)=b^{2}$.
\item[(x)] $ \frac{a}{g}+\frac{b}{f}=1.$
\item[(xi)] $(f-b)(g-a)=a(b-a).$
\item[(xii)] $ (f-b)(g-b)=b(b-a).$
\item[(xiii)] $ f(a-g)=ab$
\item[(xiv)] $f(b-g)=b^{2}. $
\item[(xv)] $ f(a-g)=a^{2}$
\item[(xvi)] $ f(b-g)=(b-1)$
\item[(xvii)] $fg=1.$
\item[(xviii)] $ f+g=2$
\item[(xix)] $ \frac{f}{a}+\frac{g}{b}=1.$
\item[(xx)] $ \frac{f}{b}+\frac{g}{a}=1.$
\item[(xxi)] $ f \equiv g.$
\end{itemize}
\end{theorem}
\section{Lemmas}
Let $f$ and $g$ be two non constant meromorphic functions. Let us define $H$ as 
$$H=\left(\frac{f''}{f'}-\frac{2f'}{f-1}\right)-\left(\frac{g''}{g'}-\frac{2g'}{g-1}\right).$$
\begin{lem}[\cite{il}]\label{lem1}
If $f, g$ share $(1,1)$ and $H \not \equiv 0$, then 
\begin{itemize}
\item[(i)] $\overline{N}(r,1;f|=1) \leq N(r,H)+S(r,f)+S(r,g).$
\item[(ii)] $\overline{N}(r,1;g|=1) \leq N(r,H)+S(r,f)+S(r,g).$
\end{itemize}
\end{lem}
\begin{lem}\label{lem2}
Let $f,g$ share $(1,0)$ and $H \not \equiv 0$. Then for any complex numbers $a(\neq 0,1,\infty),~b(\neq 0, 1, \infty)$ and $a \neq b$, we have 
\begin{eqnarray*}
N(r,H) & \leq & \overline{N}(r, \infty;f|\geq 2)+ \overline{N}(r, 0;f|\geq 2)+ \overline{N}(r, \infty;g|\geq 2)+ \overline{N}(r, 0;g|\geq 2)\\
& & +\overline{N}(r, a;f|\geq 2)+ \overline{N}(r, b;f|\geq 2)+ \overline{N}_{*}(r, 1; f, g)+ \overline{N}_{1}(r, 0;f')\\
& & + \overline{N}_{0}(r, 0;g'),
\end{eqnarray*}
where $\overline{N}_{1}(r, 0;f')$ is the reduced counting function of the zeros of $f'$, which do not come from the zeros of $f(f-1)(f-a)(f-b)$ and $\overline{N}_{0}(r, 0;f')$ is the reduced counting function of the zeros of $f'$, which do not come from the zeros of $f(f-1)$. Similarly $\overline{N}_{1}(r, 0;g')$ and $\overline{N}_{0}(r, 0;g')$ are defined.
\end{lem}
\begin{proof}
Given $$H=\left(\frac{f''}{f'}-\frac{2f'}{f-1}\right)-\left(\frac{g''}{g'}-\frac{2g'}{g-1}\right).$$
We see that the poles of $H$ comes from
\begin{itemize}
\item[(i)] poles of $f$ and $g$,
\item[(ii)] zeros of $f'$ and $g'$ and 
\item[(iii)] zeros of $f-1$ and $g-1$.
\end{itemize}
Also we see that zeros of $f'$ can be divided into two parts, firstly, the zeros of $f'$ coming from the zeros of $f$, $a, b$-points of $f$, $1$-points of $f$ and secondly the zeros of $f'$ not coming from the above points of $f$. Here $\overline{N}_{1}(r, 0;f')$ is the counting function of the zeros of $f'$, which do not come from the zeros of $f(f-1)(f-a)(f-b)$. Similarly we have divided the zeros of $g'$ into two parts, firstly those coming from the $0, 1$ -points of $g$ and secondly those not coming from the $0, 1$ -points of $g$. Also here we have taken $\overline{N}_{0}(r, 0;g')$ to be the reduced counting function of the zeros of $g'$, which do not come from the zeros of $g(g-1)$.\par
We also have seen that the simple zeros, $1$-points, simple $a$-points, simple $b$-points, simple poles of $f$, do not contribute to the poles of $H$ and the same situation is for the function $g$. Also, we have seen that when the  multiplicities of zeros of $f-1$ and $g-1$ are equal, then it also do not contribute to the poles of $H$. Therefore, considering all these, we can say that poles of $H$ can only come from
\begin{itemize}
\item[(i)] multiple poles of $f$ and $g$,
\item[(ii)] multiple $a, b$-points of $f$
\item[(iii)] multiple zeros of $f$ and $g$.
\item[(iv)] zeros of $f-1$ and $g-1$, where the multiplicities are different,
\item[(v)] zeros of $f'$ not coming from the $0, 1, a, b$-points of $f$ and
\item[(vi)] zeros of $g'$ not coming from the $0, 1$-points of $g$.
\end{itemize}
Therefore, 
\begin{eqnarray*}
N(r,H) & \leq & \overline{N}(r, \infty;f|\geq 2)+ \overline{N}(r, 0;f|\geq 2)+ \overline{N}(r, \infty;g|\geq 2)+ \overline{N}(r, 0;g|\geq 2)\\
& & +\overline{N}(r, a;f|\geq 2)+ \overline{N}(r, b;f|\geq 2)+ \overline{N}_{*}(r, 1; f, g)+ \overline{N}_{1}(r, 0;f')\\
& & + \overline{N}_{0}(r, 0;g').
\end{eqnarray*}
\end{proof}

\section{Proof of the Theorem}
\begin{proof}{\textbf{\textit{(Proof of Theorem \ref{th1})}}}
Let, $H \not \equiv 0$.\\
Since $f$ and $g$ share $1$ with weight $2.$ Then by Lemma (\ref{lem1}) and Lemma (\ref{lem2}) we have 
\begin{eqnarray}
\nonumber N(r, 1;f|=1) &\leq & N(r, H)+ S(r,f)+ S(r,g) \\
\nonumber & \leq & \overline{N}(r, \infty;f|\geq 2)+ \overline{N}(r, 0;f|\geq 2)+ \overline{N}(r, \infty;g|\geq 2)+ \overline{N}(r, 0;g|\geq 2)\\
\nonumber & & +\overline{N}(r, a;f|\geq 2)+ \overline{N}(r, b;f|\geq 2)+ \overline{N}_{*}(r, 1; f, g)+ \overline{N}_{1}(r, 0;f')\\
\label{pfeq1}& & + \overline{N}_{0}(r, 0;g')+S(r,f)+ S(r,g),
\end{eqnarray}
where $\overline{N}_{1}(r, 0;f')$ is the counting function of the zeros of $f'$, which do not come from the zeros of $f(f-1)(f-a)(f-b)$ and $\overline{N}_{0}(r, 0;g')$ is the counting function of the zeros of $g'$, which do not come from the zeros of $g(g-1)$.\\ 
Using Lemma 4 of (\cite{il}), from (\ref{pfeq1}) we obtain
\begin{eqnarray}
\nonumber N(r, 1;f|=1) &\leq & \overline{N}(r, \infty;f|\geq 2)+ \overline{N}(r, 0;f|\geq 2)+ \overline{N}(r, \infty;g|\geq 2)+ \overline{N}(r, 0;g|\geq 2)\\
\nonumber & & +\overline{N}(r, a;f|\geq 2)+ \overline{N}(r, b;f|\geq 2)+ \overline{N}_{1}(r, 0;f')+ \overline{N}(r, \infty,g)+\overline{N}(r, 0,g)\\
\label{pfeq2}& &-\overline{N}(r, 1;g|\geq 2)+ S(r,f)+S(r,g).
\end{eqnarray}
By Second Fundamental theorem we get
\begin{eqnarray}
\label{pfeq3}3T(r,f) & \leq & \overline{N}(r,f)+ \overline{N}(r,1;f)+ \overline{N}(r,0;f)+\overline{N}(r,a;f)+\overline{N}(r,b;f)\\
\nonumber & & -N_{1}(r,0;f')+S(r,f)
\end{eqnarray}
Since $f$ and $g$ share $(1,2)$, then 
\begin{eqnarray}
\label{pfeq4}\overline{N}(r,1;f)& = &\overline{N}(r,1;f|=1)+\overline{N}(r,1;f|\geq 2)\\
\nonumber &= & \overline{N}(r,1;f|=1)+ \overline{N}(r,1;g|\geq 2).
\end{eqnarray}
Using (\ref{pfeq2}), (\ref{pfeq3}) and (\ref{pfeq4}) we obtain 
\begin{eqnarray}
3T(r,f) & \leq & N_{2}(r, \infty;f)+ N_{2}(r, \infty;g)+ N_{2}(r, 0;f)+N_{2}(r, 0;g)+ N_{2}(r, a;f)\\
\nonumber & & +N_{2}(r, b;f)+ S(r,f)+ S(r,g).
\end{eqnarray}
Similarly, 
\begin{eqnarray}
3T(r,g) & \leq & N_{2}(r, \infty;f)+ N_{2}(r, \infty;g)+ N_{2}(r, 0;f)+N_{2}(r, 0;g)+ N_{2}(r, a;g)\\
\nonumber & & +N_{2}(r, b;g)+ S(r,f)+ S(r,g).
\end{eqnarray}
Using above two results, we can obtain $(i)$.\\
Let us now assume that $H=0$. Then, 
\begin{eqnarray}\label{pfeq5}
f=\frac{Ag+B}{Cg+D},
\end{eqnarray}
where $(AD-BC) \neq 0.$\par
Obviously 
\begin{eqnarray}\label{pfeq6}
T(r,f)=T(r,g)+O(1).
\end{eqnarray}
\textbf{Case:1} Let $AC \neq 0.$ Then
$$f-\frac{A}{C}=\frac{B-\frac{AD}{C}}{Cg+D}.$$
\textbf{Subcase:1.1} Let, $\frac{A}{C}\neq a,b.$
Using Second Fundamental theorem we get,
\begin{eqnarray}
\label{pfeq7}3T(r,f) & \leq & \overline{N}(r,f)+ \overline{N}(r,0;f)+ \overline{N}(r,a;f)+\overline{N}(r,b;f)+\overline{N}(r,\frac{A}{C};f)+S(r,f). \\
\nonumber & \leq & \overline{N}(r,f)+ \overline{N}(r,0;f)+ \overline{N}(r,a;f)+\overline{N}(r,b;f)+\overline{N}(r,g)+S(r,f).
\end{eqnarray}
Using (\ref{pfeq6}) and (\ref{pfeq7}) we can easily obtain $(i)$.  \\
\textbf{Subcase:1.2} Let, $\frac{A}{C}= a$ and $\frac{A}{C} \neq b.$ \\
\textbf{Subcase:1.2.1} Let, $BD=0.$\\
\textbf{Subcase:1.2.1.1} Let, $B \neq 0, D=0$. Then 
\begin{eqnarray}
\label{pfeq7.1}f = a +\frac{\gamma_0}{g},
\end{eqnarray}
where $\gamma_0= \frac{B}{C}.$ \\
Let $1$ be an evP of $f$ and so of $g$, then using second fundamental theorem we get
\begin{eqnarray}
\nonumber \label{pfeq8}3T(r,f) & \leq & \overline{N}(r,f)+ \overline{N}(r,0;f)+ \overline{N}(r,a;f)+\overline{N}(r,b;f)+\overline{N}(r,1;f)+S(r,f). \\
\nonumber & \leq & \overline{N}(r,f)+ \overline{N}(r,0;f)+ \overline{N}(r,a;f)+\overline{N}(r,b;f)+S(r,f),
\end{eqnarray}
which using (\ref{pfeq6}), easily obtains $(i)$.\\
Let, $1$ be not an evP of $f$ and $g$. Therefore, from (\ref{pfeq7.1}), we obtain, 
$$ f = a+\frac{1-a}{g}. $$
Then if $a \neq 1-\frac{1}{a}$ and $b \neq 1-\frac{1}{a}$, then using second fundamental theorem 
\begin{eqnarray}
\nonumber \label{pfeq9}3T(r,g) & \leq & \overline{N}(r,g)+ \overline{N}(r,0;g)+ \overline{N}(r,a;g)+\overline{N}(r,b;g)+\overline{N}(r,1-\frac{1}{a};g)+S(r,g). \\
\nonumber & \leq & \overline{N}(r,g)+ \overline{N}(r,0;g)+ \overline{N}(r,a;g)+\overline{N}(r,b;g)+\overline{N}(r,0;f)+S(r,g),
\end{eqnarray}
which using (\ref{pfeq6}), easily obtains $(i)$.\\
If $a = 1-\frac{1}{a}$ and $b \neq 1-\frac{1}{a}$, then $a= -\omega, -\omega^{2}.$, which eventually implies $ g(a-f) \equiv a^2.$\\
Also if $a \neq 1-\frac{1}{a}$ and $b = 1-\frac{1}{a}$, then it leads us to $g(a-f)=ab$.\\
\textbf{Subcase:1.2.1.2} Let, $B=0, D \neq 0$. Then 
\begin{eqnarray}
\label{pfeq15} f-a= - \frac{a}{1+\beta_0 g}, 
\end{eqnarray} 
where $\beta_0 = \frac{C}{D}.$ \par
Let, $1$ be an evP of $f$ and then so of $g$. Therefore, by second fundamental theorem, we obtain 
\begin{eqnarray}
\nonumber 3T(r,f) & \leq & \overline{N}(r,f)+ \overline{N}(r,0;f)+ \overline{N}(r,a;f)+\overline{N}(r,b;f)+\overline{N}(r,1;f)+S(r,f)\\
& = & \overline{N}(r,f)+ \overline{N}(r,0;f)+ \overline{N}(r,a;f)+\overline{N}(r,b;f)+S(r,f).
\end{eqnarray}
From here, using (\ref{pfeq6}), we can easily obtain $(i)$.  \\
Let $1$ be not an evP of $f$ and $g$. Thus from (\ref{pfeq15}) we get, $$f= \frac{ag}{(a-1)+g}.$$ 
Then if $a \neq \frac{1}{2}, a+b\neq 1$, using second fundamental theorem we obtain, 
\begin{eqnarray}
\nonumber 3T(r,g) & \leq & \overline{N}(r,g)+ \overline{N}(r,0;g)+ \overline{N}(r,a;g)+\overline{N}(r,b;g)+\overline{N}(r,1-a;g)+S(r,g).\\
\nonumber & = & \overline{N}(r,g)+ \overline{N}(r,0;g)+ \overline{N}(r,a;g)+\overline{N}(r,b;g)+\overline{N}(r,f)+S(r,g).
\end{eqnarray} 
From here, using (\ref{pfeq6}), we can easily obtain $(i)$.  \\
If $a = \frac{1}{2}, a+b\neq 1$, then $$\frac{1}{f}+\frac{1}{g}=2.$$\\
If $a \neq \frac{1}{2}, a+b = 1$, then $$\frac{a}{f}+\frac{b}{g}=1.$$\\
\textbf{Subcase:1.2.2} Let, $BD \neq 0$. Then $\frac{B}{D}\neq a$, since $(AD-BC) \neq 0$.\\
Then if $\frac{B}{D} \neq b,$. Then by Second Fundamental theorem we get 
\begin{eqnarray}
\nonumber 3T(r,f) & \leq & \overline{N}(r,f)+ \overline{N}(r,0;f)+ \overline{N}(r,a;f)+\overline{N}(r,b;f)+\overline{N}(r,\frac{B}{D};f)+S(r,f).\\
\nonumber & = & \overline{N}(r,f)+ \overline{N}(r,0;f)+ \overline{N}(r,a;f)+\overline{N}(r,b;f)+\overline{N}(r, 0;g)+S(r,f).
\end{eqnarray}
From here, using (\ref{pfeq6}), we can easily obtain $(i)$.  \\
If $\frac{B}{D} = b$, then from (\ref{pfeq5}), we get 
\begin{eqnarray}
\nonumber f- \frac{A}{C} & = & \left(\frac{B}{D}\right)\left\{\frac{1-(\frac{A}{C})(\frac{D}{B})}{1+(\frac{C}{D})g}\right\}\\
\label{pfeq16}\implies (f-a) &=& b \left\{\frac{1-(\frac{a}{b})}{1+\alpha_0 g}\right\},
\end{eqnarray}
where $\alpha_0= \frac{C}{D}$. \\
Let, $1$ be an evP of $f$ and $g$, then by Second Fundamental theorem we get 
\begin{eqnarray}
\nonumber 3T(r,f) & \leq & \overline{N}(r,f)+ \overline{N}(r,0;f)+ \overline{N}(r,a;f)+\overline{N}(r,b;f)+\overline{N}(r,1;f)+S(r,f)\\
\nonumber & \leq & \overline{N}(r,f)+ \overline{N}(r,0;f)+ \overline{N}(r,a;f)+\overline{N}(r,b;f)+S(r,f),
\end{eqnarray} 
using (\ref{pfeq6}), we can easily obtain $(i)$.\\
Let $1$ be not an evP of $f$ and $g$, then from (\ref{pfeq16}), we get that $\alpha_0 = \frac{1-b}{a-1}$. Then 
$$ f-a = \frac{(a-b)(a-1)}{(b-1)}\left\{\frac{1}{g+(\frac{a-1}{1-b})}\right\}.$$ 
If $(\frac{a-1}{b-1}) \neq a, b$, then from Second Fundamental theorem we obtain  
\begin{eqnarray}
\nonumber 3T(r,g) & \leq & \overline{N}(r,g)+ \overline{N}(r,0;g)+ \overline{N}(r,a;g)+\overline{N}(r,b;g)+\overline{N}(r,\small{(\frac{a-1}{b-1})};g)+S(r,f)\\
\nonumber & \leq & \overline{N}(r,g)+ \overline{N}(r,0;g)+ \overline{N}(r,a;g)+\overline{N}(r,b;g)+\overline{N}(r,f)+S(r,f),
\end{eqnarray}
using (\ref{pfeq6}), we can easily obtain $(i)$.\\
If $(\frac{a-1}{b-1}) = a$, then $(f-a)(g-a)=a(a-b).$ Also if $(\frac{a-1}{b-1}) = b$, then $(f-a)(g-b)=b(a-b).$\\
\textbf{Subcase:1.3} Let, $\frac{A}{C}= b$ and $\frac{A}{C} \neq a.$ Then $\frac{B}{D}\neq b$, since $(AD-BC) \neq 0$.\\
\textbf{Subcase:1.3.1} Let, $BD=0.$\\
\textbf{Subcase:1.3.1.1} Let, $B \neq 0, D=0$. Therefore $$f=b+\frac{\xi}{g},$$
where $\xi = \frac{B}{C}.$\\
If $1$ be an evP of $f$ and $g$, then using second fundamental theorem we get, 
\begin{eqnarray}
\nonumber 3T(r,f) & \leq & \overline{N}(r,f)+ \overline{N}(r,0;f)+ \overline{N}(r,a;f)+\overline{N}(r,b;f)+\overline{N}(r,1;f)+S(r,f)\\
\nonumber& = & \overline{N}(r,f)+ \overline{N}(r,0;f)+ \overline{N}(r,a;f)+\overline{N}(r,b;f)+S(r,f).
\end{eqnarray}
From here, using (\ref{pfeq6}), we can easily obtain $(i)$. \\
Let, $1$ is not an evP of $f$ and $g$. Consequently we get $$f=b+\frac{(1-b)}{g}.$$\\
Then if $a \neq 1-\frac{1}{b}$ and $b \neq 1-\frac{1}{b}$, then using second fundamental theorem we get, 
\begin{eqnarray}
\nonumber 3T(r,g) & \leq & \overline{N}(r,g)+ \overline{N}(r,0;g)+ \overline{N}(r,a;g)+\overline{N}(r,b;g)+\overline{N}(r,1-\frac{1}{b};g)+S(r,g)\\
\nonumber& = & \overline{N}(r,g)+ \overline{N}(r,0;g)+ \overline{N}(r,a;g)+\overline{N}(r,b;g)+\overline{N}(r,0;f)+S(r,g).
\end{eqnarray}
From here, using (\ref{pfeq6}), we can easily obtain $(i)$. \\
If $a = 1-\frac{1}{b}$ and $b \neq 1-\frac{1}{b}$, then $g(b-f)=ab.$\\
Also if $a \neq 1-\frac{1}{b}$ and $b = 1-\frac{1}{b}$, then $g(b-f)=b^{2}.$\\
\textbf{Subcase:1.3.1.2} Let, $B=0, D \neq 0$. Therefore, $$ f-b=- \frac{b}{1+ \eta g},$$ 
where $\eta = \frac{C}{D}.$\\
If $1$ is an evP of $f$ and $g$, then using second fundamental theorem we get, 
\begin{eqnarray}
\nonumber 3T(r,f) & \leq & \overline{N}(r,f)+ \overline{N}(r,0;f)+ \overline{N}(r,a;f)+\overline{N}(r,b;f)+\overline{N}(r,1;f)+S(r,f)\\
\nonumber& = & \overline{N}(r,f)+ \overline{N}(r,0;f)+ \overline{N}(r,a;f)+\overline{N}(r,b;f)+S(r,f).
\end{eqnarray}
From here, using (\ref{pfeq6}), we can easily obtain $(i)$. \\
Let $1$ is not an evP of $f$ and $g$. Therefore, $$ f=\frac{bg}{g+(b-1)}.$$ 
Then if $b \neq \frac{1}{2}, a+b\neq 1$, then using second fundamental theorem we get, 
\begin{eqnarray}
\nonumber 3T(r,g) & \leq & \overline{N}(r,g)+ \overline{N}(r,0;g)+ \overline{N}(r,a;g)+\overline{N}(r,b;g)+\overline{N}(r,1-b;g)+S(r,g)\\
\nonumber & = & \overline{N}(r,g)+ \overline{N}(r,0;g)+ \overline{N}(r,a;g)+\overline{N}(r,b;g)+\overline{N}(r,f)+S(r,g).
\end{eqnarray}
From here, using (\ref{pfeq6}), we can easily obtain $(i)$. \\
If $b = \frac{1}{2}, a+b\neq 1$, then $$\frac{1}{f}+\frac{1}{g}=2.$$\\.\\
Also if $b \neq \frac{1}{2}, a+b = 1$, then $$ \frac{a}{g}+\frac{b}{f}=1.$$ \\
\textbf{Subcase:1.3.2} Let, $BD \neq 0$. Then if $\frac{B}{D} \neq a,$ Then by Second Fundamental theorem we get 
\begin{eqnarray}
\nonumber 3T(r,f) & \leq & \overline{N}(r,f)+ \overline{N}(r,0;f)+ \overline{N}(r,a;f)+\overline{N}(r,b;f)+\overline{N}(r,\frac{B}{D};f)+S(r,f).\\
\nonumber & = & \overline{N}(r,f)+ \overline{N}(r,0;f)+ \overline{N}(r,a;f)+\overline{N}(r,b;f)+\overline{N}(r, 0;g)+S(r,f).
\end{eqnarray}
From here, using (\ref{pfeq6}), we can easily obtain $(i)$.  \\
If $\frac{B}{D} = a$, then from (\ref{pfeq5}), we get 
\begin{eqnarray}
\nonumber f- \frac{A}{C} & = & \left(\frac{B}{D}\right)\left\{\frac{1-(\frac{A}{C})(\frac{D}{B})}{1+(\frac{C}{D})g}\right\}\\
\label{pfeq17}\implies (f-b) &=& a \left\{\frac{1-(\frac{b}{a})}{1+\alpha_1 g}\right\},
\end{eqnarray}
where $\alpha_1 = \frac{C}{D}$. \\
If $1$ be an evP of $f$ and $g$, then by Second Fundamental theorem we get 
\begin{eqnarray}
\nonumber 3T(r,f) & \leq & \overline{N}(r,f)+ \overline{N}(r,0;f)+ \overline{N}(r,a;f)+\overline{N}(r,b;f)+\overline{N}(r,1;f)+S(r,f)\\
\nonumber & \leq & \overline{N}(r,f)+ \overline{N}(r,0;f)+ \overline{N}(r,a;f)+\overline{N}(r,b;f)+S(r,f),
\end{eqnarray} 
using (\ref{pfeq6}), we can easily obtain $(i)$.\\
Let $1$ be not an evP of $f$ and $g$, then from (\ref{pfeq17}), we get that $\alpha_1 = \frac{1-a}{b-1}$. Then 
$$ f-b = \frac{(b-a)(b-1)}{(a-1)}\left\{\frac{1}{g+(\frac{b-1}{1-a})}\right\}.$$ 
If $(\frac{b-1}{a-1}) \neq a, b$, then from Second Fundamental theorem we obtain  
\begin{eqnarray}
\nonumber 3T(r,g) & \leq & \overline{N}(r,g)+ \overline{N}(r,0;g)+ \overline{N}(r,a;g)+\overline{N}(r,b;g)+\overline{N}(r,\small{(\frac{b-1}{a-1})};g)+S(r,f)\\
\nonumber & \leq & \overline{N}(r,g)+ \overline{N}(r,0;g)+ \overline{N}(r,a;g)+\overline{N}(r,b;g)+\overline{N}(r,f)+S(r,f),
\end{eqnarray}
using (\ref{pfeq6}), we can easily obtain $(i)$.\\
If $(\frac{b-1}{a-1}) = a$, then $(f-b)(g-a)=a(b-a).$ Also if $(\frac{b-1}{a-1}) = b$, then $(f-b)(g-b)=b(b-a).$\\
\textbf{Case:2} Let $AC=0.$ Since $f$ is non-constant, then it follows that $A$ and $C$ are not simultaneously zero. Thus we consider \\
\textbf{Subcase:2.1} Let $A=0, C \neq 0$. Then from(\ref{pfeq5}), we get that 
$$f=\frac{1}{\alpha g+ \beta},$$
where $\alpha = \frac{C}{B}$ and $\beta = \frac{D}{B}$.\\
\textbf{Subcase:2.1.1} Let, $1$ be an evP of $f$ and $g$, then by second fundamental theorem
\begin{eqnarray}
\nonumber 3T(r,f) & \leq & \overline{N}(r,f)+ \overline{N}(r,0;f)+ \overline{N}(r,a;f)+\overline{N}(r,b;f)+\overline{N}(r,1;f)+S(r,f)\\
\nonumber & \leq & \overline{N}(r,f)+ \overline{N}(r,0;f)+ \overline{N}(r,a;f)+\overline{N}(r,b;f)+S(r,f),
\end{eqnarray}
using (\ref{pfeq6}), we can easily obtain $(i)$.\\
\textbf{Subcase:2.1.2} Let, $1$ be not an evP of $f$ and $g$. Then $\alpha g =\frac{1}{f}-(1-\alpha)$.\\
\textbf{Subcase:2.1.2.1} If $\alpha \neq 1, \alpha \neq 1-\frac{1}{a}, \alpha \neq 1-\frac{1}{b}$, then using second fundamental theorem, we obtain that 
\begin{eqnarray}
\nonumber 3T(r,f) & \leq & \overline{N}(r,f)+ \overline{N}(r,0;f)+ \overline{N}(r,a;f)+\overline{N}(r,b;f)+\overline{N}(r,\frac{1}{1- \alpha};f)+S(r,f)\\
\nonumber & \leq & \overline{N}(r,f)+ \overline{N}(r,0;f)+ \overline{N}(r,a;f)+\overline{N}(r,b;f)+ \overline{N}(r,0;g)+S(r,f),
\end{eqnarray}
using (\ref{pfeq6}), we can easily obtain $(i)$.\\
\textbf{Subcase:2.1.2.2} If $\alpha \neq 1, \alpha \neq 1-\frac{1}{a}, \alpha = 1-\frac{1}{b}$, therefore,$$f=\frac{b}{(b-1)g+1}.$$  
Then if $a \neq \frac{1}{1-b}, b \neq \frac{1}{1-b}$, then using second fundamental theorem we obtain 
\begin{eqnarray}
\nonumber 3T(r,g) & \leq & \overline{N}(r,g)+ \overline{N}(r,0;g)+ \overline{N}(r,a;g)+\overline{N}(r,b;g)+\overline{N}(r,\frac{1}{1- b};g)+S(r,g)\\
\nonumber & \leq & \overline{N}(r,g)+ \overline{N}(r,0;g)+ \overline{N}(r,a;g)+\overline{N}(r,b;g)+ \overline{N}(r,f)+S(r,g),
\end{eqnarray}
using (\ref{pfeq6}), we can easily obtain $(i)$.\\
If $a = \frac{1}{1-b}, b \neq \frac{1}{1-b}$, then $f(a-g)=ab.$\\
Also if $a \neq \frac{1}{1-b}, b = \frac{1}{1-b}$, then $f(b-g)=b^2$.\\
\textbf{Subcase:2.1.2.3} If $\alpha \neq 1, \alpha = 1-\frac{1}{a}, \alpha \neq 1-\frac{1}{b}$, then $$f= \frac{a}{1+(a-1)g}.$$ 
If $a \neq \frac{1}{1-a}, b \neq \frac{1}{1-a}$, then using second fundamental theorem we get 
\begin{eqnarray}
\nonumber 3T(r,g) & \leq & \overline{N}(r,g)+ \overline{N}(r,0;g)+ \overline{N}(r,a;g)+\overline{N}(r,b;g)+\overline{N}(r,\frac{1}{1-a};g)+S(r,g)\\
\nonumber & \leq & \overline{N}(r,g)+ \overline{N}(r,0;g)+ \overline{N}(r,a;g)+\overline{N}(r,b;g)+ \overline{N}(r,f)+S(r,g),
\end{eqnarray}
using (\ref{pfeq6}), we can easily obtain $(i)$.\\
If $a = \frac{1}{1-a}, b \neq \frac{1}{1-a}$, then $f(a-g)=a^{2}$.\\
Also if $a \neq \frac{1}{1-a}, b = \frac{1}{1-a}$, then $f(g-b)=(1-b).$\\
\textbf{Subcase:2.1.2.4} If $\alpha = 1$, then  obviously $ \alpha \neq 1-\frac{1}{a}, \alpha \neq 1-\frac{1}{b}$, so we have $fg \equiv 1$.\\
\textbf{Subcase:2.2} Let $A\neq 0, C= 0$. Then from(\ref{pfeq5}), we get that 
$$f=\gamma g + \delta,$$
where $\gamma = \frac{A}{D}$ and $\delta = \frac{B}{D}$.\\
\textbf{Subcase:2.2.1} Let, $1$ be an evP of $f$ and $g$, then by second fundamental theorem we obtain 
\begin{eqnarray}
\nonumber 3T(r,f) & \leq & \overline{N}(r,f)+ \overline{N}(r,0;f)+ \overline{N}(r,a;f)+\overline{N}(r,b;f)+\overline{N}(r,1;f)+S(r,f)\\
\nonumber & \leq & \overline{N}(r,f)+ \overline{N}(r,0;f)+ \overline{N}(r,a;f)+\overline{N}(r,b;f)+S(r,f),
\end{eqnarray}
using (\ref{pfeq6}), we can easily obtain $(i)$.\\
\textbf{Subcase:2.2.2} Let, $1$ be not an evP of $f$ and $g$. Then $f = \gamma g +(1- \gamma).$ Therefore we consider the following cases:\\
\textbf{Subcase:2.2.2.1} If $\gamma \neq 1, \gamma \neq 1-a, \gamma \neq 1-b$, then using second fundamental theorem we get 
\begin{eqnarray}
\nonumber 3T(r,f) & \leq & \overline{N}(r,f)+ \overline{N}(r,0;f)+ \overline{N}(r,a;f)+\overline{N}(r,b;f)+\overline{N}(r,1-\gamma;f)+S(r,f)\\
\nonumber & \leq & \overline{N}(r,f)+ \overline{N}(r,0;f)+ \overline{N}(r,a;f)+\overline{N}(r,b;f)+\overline{N}(r,0;g)+S(r,f),
\end{eqnarray}
using (\ref{pfeq6}), we can easily obtain $(i)$.\\
\textbf{Subcase:2.2.2.2} If $\gamma \neq 1, \gamma = 1-a, \gamma \neq 1-b$, then $f= (1-a)g+a.$ Therefore,\\
then if $a \neq \frac{a}{a-1}, b \neq \frac{a}{a-1}$, then using second fundamental theorem, 
\begin{eqnarray}
\nonumber 3T(r,g) & \leq & \overline{N}(r,g)+ \overline{N}(r,0;g)+ \overline{N}(r,a;g)+\overline{N}(r,b;g)+\overline{N}(r,\frac{a}{a-1};g)+S(r,g)\\
\nonumber & \leq & \overline{N}(r,g)+ \overline{N}(r,0;g)+ \overline{N}(r,a;g)+\overline{N}(r,b;g)+\overline{N}(r,0;f)+S(r,g),
\end{eqnarray}
using (\ref{pfeq6}), we can easily obtain $(i)$.\\
If $a = \frac{a}{a-1}, b \neq \frac{a}{a-1}$. then $f+g =2$. \\
Also if $a \neq \frac{a}{a-1}, b = \frac{a}{a-1}$, then $\frac{f}{a}+\frac{g}{b}=1.$\\
\textbf{Subcase:2.2.2.3} If $\gamma \neq 1, \gamma \neq 1-a, \gamma = 1-b$, therefore, $f= (1-b)g+b.$.\\ Then if $a \neq \frac{b}{b-1}, b \neq \frac{b}{b-1}$, then using second fundamental theorem, 
\begin{eqnarray}
\nonumber 3T(r,g) & \leq & \overline{N}(r,g)+ \overline{N}(r,0;g)+ \overline{N}(r,a;g)+\overline{N}(r,b;g)+\overline{N}(r,\frac{b}{b-1};g)+S(r,g)\\
\nonumber & \leq & \overline{N}(r,g)+ \overline{N}(r,0;g)+ \overline{N}(r,a;g)+\overline{N}(r,b;g)+\overline{N}(r,0;f)+S(r,g),
\end{eqnarray}
using (\ref{pfeq6}), we can easily obtain $(i)$.\\
If $a = \frac{b}{b-1}, b \neq \frac{b}{b-1}$, then $\frac{f}{b}+\frac{g}{a}=1.$\\
Also if $a \neq \frac{b}{b-1}, b = \frac{b}{b-1}$, then $f+g =2$.\\
\textbf{Subcase:2.2.2.4} If $\gamma = 1,$  then obviously $\gamma \neq 1-a, \gamma \neq 1-b$, so, we obtain $f \equiv g.$\\
\end{proof}
\section{Conflict of Interest Declaration}
The authors have no conflicts of interest to declare. All co-authors have seen and agree with the contents of the manuscript and there is no financial interest to report. We certify that the submission is original work and is not under review at any other publication.
\section{Acknowledgements}
The authors are grateful to the anonymous referees for their valuable suggestions which considerably improved the presentation of the paper.\par
Mr. Sudip Saha is thankful to the Council of Scientific and Industrial Research, HRDG, India for granting Senior Research
Fellowship (File No.: 08/525(0003)/2019-EMR-I) during the tenure of which this work was done.\par 
 Mr. Amit Kumar Pal is thankful to the Higher Education Dept., Govt. of West Bengal, for granting the SVMCM research fellowship(No.: WBP211653035749), during the tenure of which the work was done. \par
The research work of Mr. Soumon Roy is  supported by the Department of Higher Education, Science and Technology \text{\&} Biotechnology, Govt. of West Bengal under the sanction order no. 1303(sanc.)/STBT-11012(26)/17/2021-ST SEC dated 14/03/2022.


\begin{thebibliography}{99}
\bibitem{mf} M. Fang, Uniqueness of admissible meromorphic functions in the unit disc, Sciences in China (Series A), 42(4), 1999, 367-381.
\bibitem{wkh} W. K. Hayman, Meromorphic Functions, The Clarendon Press, Oxford(1964).
\bibitem{xhh} X. H. Hua, Sharing values and a question of C. C. Yang, Pacific J. Math. 175(1), 1996, 71-81.
\bibitem{il} I. Lahiri, Weighted Value Sharing and Uniqueness of Meromorphic Functions, Complex Variables, Vol 46, 241-253, 2001.
\bibitem{mr} E. Mues and M. Reinders, Meromorphic functions sharing one value and unique range sets, Kodai Math J., 18, 1995, 515-522.
\bibitem{yahu} C. C. Yang and X. H. Hua, Uniqueness and value sharing of meromorphic functions, Ann. Acad. Sci. Fenn. Math. , 22, 1997, 395 -406.
\bibitem{yi} H. X. Yi, Meromorphic function that share one or two values, Complex Variable Theory Appl., 28, 1995,1 1-11.
\end{thebibliography}
\end{document}